\documentclass[12pt, reqno]{amsart}
\usepackage{amssymb, amsmath, amsfonts, amscd, calligra, mathrsfs}
\usepackage[raiselinks=false,colorlinks=true,citecolor=blue,urlcolor=,
linkcolor=blue,bookmarksopen=true,pdftex]{hyperref}
\numberwithin{equation}{section}
\usepackage[dvipsnames]{xcolor}
\usepackage{tikz}\usepackage{tikz-cd}
\usepackage{enumerate}
\usepackage[mathscr]{eucal}
\newtheorem{theorem}{Theorem}[section]
\newtheorem{proposition}[theorem]{Proposition}
\newtheorem{lemma}[theorem]{Lemma}
\newtheorem{remark}[theorem]{Remark}

\newtheorem{corollary}[theorem]{Corollary}

\newcommand{\rank}{\textrm{rank}}

\newcommand{\ad}{\textrm{ad}}

\newcommand{\Aut}{\textrm{Aut}}

\newcommand{\At}{\textrm{At}}

\newcommand{\from}{\textrm{from}}
\newcommand{\Gal}{\textrm{Gal}}

\newcommand{\Ad}{\textrm{Ad}}

\usepackage{etoolbox}
\makeatletter
\patchcmd{\@tocline}{\hfil}{\dotfill}{}{}
\makeatother
\begin{document}

\title[Parabolic Lie algebroid connections]{Parabolic Lie algebroid connections on parabolic principal 
bundles over curves}

\author[I. Biswas]{Indranil Biswas} 

\address{Department of Mathematics, Shiv Nadar University, NH91, Tehsil Dadri, Greater Noida, Uttar Pradesh 
201314, India.}

\email{indranil.biswas@snu.edu.in, indranil29@gmail.com}

\author[P. Biswas]{Pritthijit Biswas} 

\address{Kerala School of Mathematics, Kunnamangalam, Kozhikode, 673571 Kerala, India.}

\email{pritthijit@ksom.res.in, pritthibis06@gmail.com}



\date{}

\begin{abstract}
Let $X$ be a compact connected Riemann surface and $S\,\subset\, X$ a finite subset. We consider parabolic 
principal $G$--bundles $\mathcal{E}_{G}$ on $X$ with parabolic structure on
$S$, where $G$ is a connected complex reductive affine algebraic group. Let $P\, \subset\, G$ be a parabolic
subgroup and $\mathcal{E}_{P}\, \subset\, \mathcal{E}_{G}$ a 
reduction of structure group of $\mathcal{E}_{G}$ to $P$. We give a criterion 
for the existence of a parabolic Lie algebroid connection on $\mathcal{E}_{P}$ for any given parabolic Lie 
algebroid on $(X,\,S)$ whose anchor map is not surjective. More precisely,
$\mathcal{E}_{P}$ admits a parabolic Lie algebroid connection if the reduction $\mathcal{E}_{P}\, \subset\,
\mathcal{E}_{G}$ is parabolically infinitesimally rigid. In particular, the 
Harder--Narasimhan reduction of $\mathcal{E}_{G}$ admits a parabolic Lie 
algebroid connection.
\end{abstract}

\maketitle
\tableofcontents
\section{Introduction}

A holomorphic vector bundle $E$ on a compact Riemann surface admits a holomorphic connection if and only if 
every indecomposable component of $E$ is of degree zero \cite{At}. This generalizes to a criterion for 
the existence of a holomorphic connection on a principal $G$--bundle, where
$G$ is a connected complex reductive affine algebraic group \cite{AB1}.

Lie algebroid connections constitute a generalization of holomorphic connections. For a detailed study of Lie 
algebroids, Lie algebroid connections and their properties, we refer the reader to \cite{AO}, \cite{BMRT}, 
\cite{BR}, \cite{To1} and \cite{To2} (these are briefly recalled in Section \ref{2.2} and Section \ref{2.3}).

We recall some known results on holomorphic Lie algebroid connections.
Given a holomorphic Lie algebroid on a compact connected Riemann surface, a complete classification
of holomorphic vector bundles admitting a holomorphic Lie algebroid connection was established
in \cite{ABKS1}. This was generalized for holomorphic principal $G$--bundles in \cite{Bi1}.
Given any holomorphic Lie algebroid whose underlying holomorphic vector bundle is stable,
\cite{ABKS} gives a complete characterization of
all the parabolic vector bundles that admit a parabolic Lie algebroid connection.

Take a holomorphic Lie algebroid $\mathcal{L}\,=\, (V,\, \phi)$ on a compact connected Riemann surface $X$ 
such that the anchor map $\phi$ is not surjective. Let $G$ be a connected complex reductive affine algebraic 
group and $E_{G}$ be a holomorphic principal $G$--bundle on $X$. Take a parabolic subgroup $P\, \subset\,G$ 
and a holomorphic reduction of structure group $E_{P}\,\subset\, E_{G}$ of $E_{G}$ to $P$. Then $E_{P}$ admits 
a holomorphic $\mathcal{L}$--connection if the reduction $E_{P}\,\subset\, E_{G}$ is infinitesimally rigid, 
i.e., $H^{0}(X,\, \ad(E_{G})/\ad(E_{P}))\,=\,0$ \cite{BBK}. Our aim here is to generalize this to the context 
of principal bundles with parabolic structure over a divisor.
 
Let $X$ be a compact connected Riemann surface and $S\,\subset\, X$ be a finite subset.
As before, $G$ is a connected complex reductive affine algebraic group. Let
$\mathcal{E}_{G}$ be a parabolic principal $G$--bundle with parabolic structure over $S$. Take
a parabolic subgroup $P\, \subset\, G$ and a parabolic reduction of structure group 
$\mathcal{E}_{P}\, \subset\, \mathcal{E}_{G}$ of $\mathcal{E}_{G}$ to $P$. This reduction is
called parabolically infinitesimally rigid if
$$H^{0}(X,\,\ad_{*}(\mathcal{E}_{G})/\ad_{*}(\mathcal{E}_{P}))\ =\ 0,$$
where $\ad_{*}(\mathcal{E}_{G})$ and $\ad_{*}(\mathcal{E}_{P})$ are the parabolic adjoint vector bundles of
$\mathcal{E}_{G}$ and $\mathcal{E}_{P}$ respectively with parabolic structures over $S$.
Let $\mathcal{L}_{*}\,=\, (V_{*},\,\phi_{*})$ be a parabolic Lie algebroid on $X$ with
parabolic structure over $S$.
We prove the following (see Theorem \ref{MAIN}):

\begin{theorem}\label{int1}
If the anchor map $\phi_{*}$ of the parabolic Lie algebroid $\mathcal{L}_{*}$ is not surjective, then
$\mathcal{E}_{P}$ admits a holomorphic parabolic $\mathcal{L}_{*}$--connection if the reduction
$\mathcal{E}_{P}\, \subset\, \mathcal{E}_{G}$ is parabolically infinitesimally rigid.
\end{theorem}

The following is deduced from Theorem \ref{int1} (see Corollary \ref{main}):

\begin{theorem}\label{int2}
Take $(X,\, S)$ and $G$ as in Theorem \ref{int1}.
Let $\mathcal{E}_{G}$ be a parabolic principal $G$--bundle on $(X,\,S)$ which is not semistable.
Let $$\mathcal{E}_{P}\ \subset\ \mathcal{E}_{G}$$ be the parabolic 
Harder--Narasimhan reduction of structure group of $\mathcal{E}_{G}$.
Then there is a holomorphic parabolic $(V_{*},\, \phi_{*})$--connection on $\mathcal{E}_{P}$
if the anchor map $\phi_{*}$ of the parabolic Lie algebroid $\mathcal{L}_{*}$ is not surjective.
\end{theorem}

A criterion for the existence of a parabolic Lie algebroid connection on a parabolic vector bundle over a 
compact connected Riemann surface with finitely many marked points is given in \cite{ABKS}. In \cite{Bi1} it 
is shown that for a non-split parabolic Lie algebroid $\mathcal{L}_{*}$ on a compact connected Riemann 
surface $X$ with a set of finitely many marked points $S$, every parabolic principal $G$--bundle 
$\mathcal{E}_{G}$ on $(X,\,S)$ where $G$ is a connected complex reductive affine algebraic group, admits a 
parabolic $\mathcal{L}_{*}$ connection.

The proof of Theorem \ref{int1} crucially uses a correspondence between the ``parabolic objects" on $(X,\,S)$ 
and the ``$\Gamma:=\Gal(\sigma)$--equivariant objects" on $Y$ for some ramified Galois cover 
$\sigma\,:\,Y\,\longrightarrow\, X$ (see Theorem \ref{char}). First an analog of Theorem \ref{int1} is
proved for ``$\Gamma$--equivariant objects" on $Y$
(see Theorem \ref{4.1}), and then extended to the parabolic case by
using Theorem \ref{char}.

\section{Preliminaries}\label{prelim}

\subsection{Notation and set-up}\label{2.1}

Let $X$ be a compact connected Riemann surface. The sheaf of holomorphic functions on $X$ will be denoted
by ${\mathcal O}_X$. The holomorphic
tangent (respectively, cotangent) bundle of $X$ will be denoted by $TX$ (respectively, $K_X$).
Denote by $\Aut(X)$ the group of all holomorphic automorphisms of $X$.
Let $$\Gamma\ \, \subset\ \, \text{Aut}(X)$$ be a finite subgroup. So $\Gamma$ has a tautological
action on $X$.

Let $G$ be a connected reductive affine algebraic group defined over the field $\mathbb C$ of complex numbers.
The Lie algebra of $G$ will be denoted by $\mathfrak g$.
Fix a Borel subgroup $B\, \subset\, G$ and a maximal torus $T\, \subset\, B$.
Take a parabolic subgroup $P\, \subset\, G$ containing $B$. The Lie algebra of $P$ will be denoted by
$\mathfrak p$.

Let
\begin{equation}\label{e1}
p\ :\ E_{G}\ \longrightarrow \ X
\end{equation}
be a $\Gamma$--equivariant holomorphic principal $G$--bundle
over $X$. This means that $\Gamma$ acts on the total space of $E_G$ via holomorphic automorphisms such that
the following two conditions hold:
\begin{enumerate}
\item The actions of $G$ and $\Gamma$ on $E_G$ commute, and

\item the projection $p$ in \eqref{e1} is $\Gamma$--equivariant (recall that $\Gamma$ acts on $X$).
\end{enumerate}
The adjoint vector bundle of $E_G$ will be denoted by $\text{ad}(E_G)$. We recall that $\text{ad}(E_G)$ is
the holomorphic vector bundle on $X$ associated to $E_G$ for the adjoint action of $G$ on its Lie
algebra $\mathfrak g$. So every fiber of $\text{ad}(E_G)$ is a Lie algebra isomorphic to $\mathfrak g$.

Let
\begin{equation}\label{ep}
E_P\ \subset\ E_G
\end{equation}
be a $\Gamma$--equivariant holomorphic reduction of structure group of $E_G$ to the parabolic subgroup $P$.
Denote by $\text{ad}(E_P)$ the adjoint vector bundle for $E_P$. So $\text{ad}(E_P)$ is
the holomorphic vector bundle on $X$ associated to $E_P$ for the adjoint action of $P$ on its Lie
algebra $\mathfrak p$; in particular, we have
$$
\text{ad}(E_P)\ \subset\ \text{ad}(E_G)
$$
as a sub--Lie algebra bundle.

The equivariant holomorphic reduction $E_{P}\,\subset \,E_{G}$ is called {\it $\Gamma$--equivariantly
infinitesimally rigid} or {\it equivariantly infinitesimally rigid} (when $\Gamma$ is clear from the context)
if we have
\begin{equation}\label{er}
H^{0}(X,\ \ad(E_{G})/\ad(E_{P}))^{\Gamma}\,\ =\,\ 0.
\end{equation}

\subsection{Equivariant Lie algebroids}\label{2.2}

Let $X$ and $\Gamma$ be as in Section \ref{2.1}. A {\it holomorphic Lie algebroid on $X$} is a pair
$\mathcal{L}\,=\,(V,\,\phi)$, where $V$ is a holomorphic vector bundle on $X$ and $\phi\,:\,V
\,\longrightarrow \,TX$ is an ${\mathcal O}_X$--linear homomorphism satisfying the following conditions:
\begin{itemize}
\item $V$ is equipped with a $\mathbb{C}$--Lie algebra structure \[[-,\, -]\ :\ V\otimes_{\mathbb{C}}V\
\longrightarrow\ V.\]

\item $[s,\,ft]\,=\,f[s,\,t]+\phi(s)(f)t$ for all locally defined holomorphic sections $s,\,t$ of $V$ and
all locally defined holomorphic functions $f$ on $X$. 
\end{itemize}
The homomorphism $\phi$ is called the {\it anchor map}. The Lie
algebroid $\mathcal{L}$ is called {\it $\Gamma$--equivariant} if the following conditions hold:
\begin{enumerate}
\item The underlying holomorphic vector bundle $V$ is a $\Gamma$--equivariant holomorphic vector bundle,

\item the $\mathbb{C}$--Lie algebra structure $[-,\, -]\, :\, V\otimes_{\mathbb{C}}V\,
\longrightarrow\, V$ is $\Gamma$--equivariant, and

\item the anchor map $\phi$ is $\Gamma$--equivariant for the action of $\Gamma$ on $TX$ given by the
action of $\Gamma$ on $X$.
\end{enumerate}

\subsection{Lie algebroid connections on equivariant principal bundles}\label{2.3}

Let $(V,\,\phi)$ be a $\Gamma$--equivariant holomorphic Lie algebroid on $X$ with $X$ as in Section \ref{2.2}.

As before, $E_{G}$ is $\Gamma$--equivariant holomorphic principal $G$--bundle on $X$.
We have the equivariant Atiyah exact sequence
\begin{equation}\label{Ateq}
0\, \longrightarrow\, \ad(E_{G})\,
\xrightarrow{\,\,\,\iota_{G}\,\,\,} \, \At(E_{G})
\,\xrightarrow{\,\,\,\omega_{G}\,\,\,}\, TX\,\longrightarrow\, 0
\end{equation}
(see \cite{At}).
Define the homomorphism
$$\psi_{G}\ :\ V\oplus\At(E_{G})\ \longrightarrow\ TX, \ \ \,
(v,\, w)\ \longmapsto\ \phi(v)-\omega_{G}(w),$$
where $\omega_G$ is the projection in \eqref{Ateq} and $\phi$ is the anchor map. Define $\mathcal{A}(E_{G})\,
:=\, {\rm Kernel}(\psi_{G})$. Note that the homomorphism
$\psi_G$ is surjective because $\omega_G$ is so. Consequently, $\mathcal{A}(E_{G})$ is a
holomorphic subbundle of $V\oplus\At(E_{G})$. As $\psi_{G}$ is
$\Gamma$--equivariant, it follows that $\mathcal{A}(E_{G})$ is a $\Gamma$--invariant subbundle of
$V\oplus\At(E_{G})$. Denote by $\pi_1\, :\, V\oplus\At(E_{G})\, \longrightarrow\, V$
and $\pi_2\, :\, V\oplus\At(E_{G})\, \longrightarrow\, \At(E_{G})$ the natural projections. Let
\begin{equation}\label{p} 
\rho_{G}\ :\ \mathcal{A}(E_{G})\ \longrightarrow\ V
\end{equation}
denote the $\Gamma$--equivariant homomorphism of holomorphic vector bundles given by the composition of maps
$$
\mathcal{A}(E_{G})\ \hookrightarrow\ V\oplus \At(E_{G})\ \xrightarrow{\,\,\,\pi_{1}\,\,\,} \ V;
$$
similarly,
\begin{equation}\label{q}
\phi_{G}\ :\ \mathcal{A}(E_{G})\ \longrightarrow\ \At(E_{G})
\end{equation}
denotes the $\Gamma$--equivariant homomorphism of holomorphic vector bundles given by the composition of maps
$$
\mathcal{A}(E_{G})\ \hookrightarrow\ V\oplus\At(E_{G}) \ \xrightarrow{\,\,\,\pi_{2}\,\,\,}\ \At(E_{G}).
$$
Then we have the following commutative diagram of $\Gamma$--equivariant holomorphic vector bundles
\begin{center}
\begin{equation}\label{d}
\begin{tikzpicture}[baseline={([yshift=-0.5ex]current bounding box.center)}]
\matrix(m)[matrix of math nodes,
row sep=2.5em, column sep=3em,
text height=1.5ex, text depth=0.25ex]
{
0 & \ad(E_{G}) & \mathcal{A}(E_{G}) & V & 0 \\
0 & \ad(E_{G}) & \At(E_{G}) & TX & 0,\\
};
\path[->]
(m-1-1) edge node[above] {} (m-1-2);
\path[->]
(m-1-2) edge node[above] {$j_{G}$} (m-1-3);
\path[->]
(m-1-3) edge node[above] {$\rho_{G}$} (m-1-4);
\path[->]
(m-1-4) edge node[above] {} (m-1-5);
\path[->]
(m-2-1) edge node[above] {} (m-2-2);
\path[->]
(m-2-2) edge node[below] {$\iota_{G}$} (m-2-3);
\path[->]
(m-2-3) edge node[below] {$\omega_{G}$} (m-2-4);
\path[->]
(m-2-4) edge node[above] {} (m-2-5);
\path[->]
(m-1-3) edge node[right] {$\phi_{G}$} (m-2-3);
\path[->]
(m-1-2) edge node[right] {${\rm Id}$} (m-2-2);
\path[->]
(m-1-4) edge node[right] {$\phi$} (m-2-4);
\end{tikzpicture}
\end{equation}
\end{center}
where \begin{equation}\label{r}
j_{G}\ =\ (0,\, \iota_{G})
\end{equation}
(see \eqref{Ateq}); note that both the rows in \eqref{d} are exact.

An {\it equivariant holomorphic $(V,\,\phi)$--connection on $E_{G}$} is a $\Gamma$--equivariant holomorphic
homomorphism of vector bundles
$$\delta\ :\ V\ \longrightarrow\ \mathcal{A}(E_{G})$$
such that $\rho_{G}\circ \delta\,=\,{\rm Id}_{V}$, where $\rho_G$ is the homomorphism in \eqref{d}.

\begin{remark}\label{rem}
An equivariant holomorphic $(V,\, \phi)$--connection on $E_{G}$ can also be equivalently viewed as a map of
equivariant holomorphic vector bundles $\delta^{'}\, :\,V\,\longrightarrow\, \At(E_{G})$ such that
$\omega_{G}\circ\delta^{'}\,=\,\phi$. See \cite{Bi1}, \cite[Lemma 2.2]{BBK}.
\end{remark}

\begin{remark}\label{rem0}
Giving a $\Gamma$--equivariant holomorphic homomorphism of vector bundles
$$\delta\ :\ V\ \longrightarrow\ \mathcal{A}(E_{G})$$ such that $\rho_{G}\circ \delta
\,=\,{\rm Id}_{V}$, where $\rho_G$ is the homomorphism in \eqref{d} is equivalent to giving a $\Gamma$--equivariant homomorphism
$$
\delta'\ :\ \mathcal{A}(E_{G}) \ \longrightarrow\ {\rm ad}(E_G)
$$
such that $\delta'\circ j_{G}\,=\, {\rm Id}_{{\rm ad}(E_G)}$, where $j_G$ is the injective homomorphism
in \eqref{d}. Given any $\delta'$ satisfying this condition, there is a unique $\Gamma$--equivariant homomorphism
$$\delta\ :\ V\ \longrightarrow\ \mathcal{A}(E_{G})$$
satisfying the following two conditions:
\begin{enumerate}
\item $\delta (V)\ =\ {\rm kernel}(\delta')$, and

\item $\rho_{G}\circ \delta\,=\,{\rm Id}_{V}$.
\end{enumerate}

Conversely, given any $\Gamma$--equivariant homomorphism $\delta\ :\ V\ \longrightarrow\ \mathcal{A}(E_{G})$
such that $\rho_{G}\circ \delta\,=\,{\rm Id}_{V}$, there is a unique $\Gamma$--equivariant homomorphism
$$
\delta'\ :\ \mathcal{A}(E_{G}) \ \longrightarrow\ {\rm ad}(E_G)
$$
satisfying the following two conditions:
\begin{enumerate}
\item $\delta (V)\ =\ {\rm kernel}(\delta')$, and

\item $\delta'\circ j_{G}\,=\, {\rm Id}_{{\rm ad}(E_G)}$, where $j_G$ is the injective homomorphism in $(\ref{d})$.
\end{enumerate}
\end{remark}

\subsection{Nilpotent radical bundle}\label{s2.4}

As before, $\mathfrak p$ denotes the Lie algebra of the parabolic subgroup $P\, \subset\, G$. The
nilpotent radical of $\mathfrak p$ will be denoted by $\mathfrak{n}$; it is the Lie algebra of the
unipotent radical $R_{u}(P)$ of $P$.

Since $G$ is reductive, there are non-degenerate symmetric $G$--invariant bilinear
forms on $\mathfrak{g}$. 
Fix a non-degenerate symmetric $G$--invariant bilinear form
$$
{\mathbb B}\ \, \in\ \, \text{Sym}^2(\mathfrak{g}^*)^G.
$$
Evidently $\mathbb B$ induces an isomorphism of $G$--modules
\begin{equation}\label{i}
\mathfrak{g}\ \overset{\cong}{\longrightarrow}\ \mathfrak{g}^{*}.
\end{equation}
Then it is straightforward to check that with respect to $\mathbb B$,
\[\mathfrak{p}^{\perp}\ =\ \mathfrak{n}.\]
Thus, using the isomorphism in \eqref{i} we see that 
\begin{equation}\label{is}
\mathfrak{g}/\mathfrak{p}\ =\ (\mathfrak{p}^{\perp})^*\ = \ \mathfrak{n}^{*}.
\end{equation}

Take a reduction $E_P$ as in \eqref{ep}. Let $\mathcal{R}(E_{P})\, \longrightarrow\, X$ be the
vector bundle associated to the principal $P$--bundle $E_P$ for the $P$--module $\mathfrak{n}$.
In other words, $\mathcal{R}(E_{P})\,\subset\, \text{ad}(E_P)$ is the nilpotent radical bundle of $E_{P}$,
which means that for each $x\, \in\, X$, the fiber $\mathcal{R}(E_{P})_x\,\subset\, \text{ad}(E_P)_x$ is the
nilpotent radical of $\text{ad}(E_P)_x$.

Using \eqref{is} it follows immediately that there is a canonical isomorphism of $\Gamma$--equivariant holomorphic
vector bundles
\begin{equation}\label{is1}
\mathcal{R}(E_{P})^{*}\ \cong\ \ad(E_{G})/\ad(E_{P}).
\end{equation}
{}From \eqref{is1} it follows that $E_P\, \subset\, E_G$ is $\Gamma$--equivariantly infinitesimally rigid
if
\begin{equation}\label{er2}
H^{0}(X,\ \mathcal{R}(E_{P})^{*})^{\Gamma}\,\ =\,\ 0~(\textrm{see}~ \eqref{er}).
\end{equation}

\section{Parabolic Lie algebroids, principal bundles and connections}\label{third}
Take a compact connected Riemann surface $X$ and fix a subset
$$S\ :=\ \{x_{1},\,x_{2},\,\cdots ,\, x_{n}\} \subset\ X$$
of $n$ distinct points of $X$, with $n\,\geq\, 1$. Let $V$ be a holomorphic vector bundle on $X$. A
{\it parabolic structure} on $V$ over $S$ is given by the following data:
\begin{itemize}
\item For every $x\,\in\, S$, there is a strictly decreasing filtration of subspaces of the fiber $V_x$ of
$V$ at $x$
\[V_{x}\,=\,V^{1}_{x}\,\supsetneq\, V^{2}_{x}\,\supsetneq\,\cdots\,
\supsetneq\, V^{l_{x}}_{x}\,\supsetneq\, V^{l_{x}+1}_{x}\,=\,0.\]
 
\item For every $x\,\in\, S$, there is a sequence of real numbers
\[0\,\leq\, \alpha^{x}_{1}\,\lneq\, \alpha^{x}_{2}\,\lneq\,\cdots\, \lneq\, \alpha^{x}_{l_{x}}\,<\,1.\]
\end{itemize}

A {\it parabolic vector bundle} on $(X,\,S)$, or a {\it parabolic vector bundle with parabolic structure over} 
$S$, is a holomorphic vector bundle $V$ on $X$ equipped with a parabolic structure as above; the notation $V_*$ 
will be used for it.

Let $V_*$, $W_*$ be two parabolic vector bundles on $(X,\,S)$. Then a {\it parabolic homomorphism}
$V_*\, \longrightarrow\, W_*$ is a holomorphic homomorphism of vector bundles
$\phi\,:\,V\,\longrightarrow\, W$ such that for every $x\,\in\, S$,
$$\phi_{x}(V^{i}_{x})\ \subseteq\ W^{j}_{x}$$ whenever $\alpha^{x}_{i}\,\geq\,\beta^{x}_{j}$, where
$\{\alpha^{x}_{k}\}$ (respectively, $\{\beta^{x}_{\ell}\}$) are the parabolic weights
of $V_*$ (respectively, $W_*$) at $x$.

Let $V_*$ be a parabolic vector bundle as above on $(X,\,S)$. Then the parabolic degree of $V_*$ which is
denoted by $\deg_{\rm par}(V_*)$, is defined to be
$$
\deg_{\rm par}(V_*)\ :=\ \deg(V)+\sum_{x\in S}\sum_{1\leq i\leq l_{x}}\alpha^{x}_{i}(\dim(V^{i}_{x})-
\dim(V^{i+1}_{x})),
$$
and the parabolic slope $\mu_{\rm par}(V_*)$ is defined to be 
$$
\mu_{\rm par}(V_*)\ :=\ \frac{\deg_{\rm par}(V_*)}{\rank(V)}.
$$
A parabolic vector bundle $V_*$ on $(X,\, S)$ is called {\it stable} (respectively,
{\it semistable}) if for all nonzero proper holomorphic subbundle $W\, \subset\ V$,
$$\mu_{\rm par}(W_*)\ <\ \mu_{\rm par}(V_*)$$
(respectively, $\mu_{\rm par}(W_*)\ \leq\ \mu_{\rm par}(V_*)$), where
$W_*$ is $W$ equipped with the parabolic structure induced by that of $V_*$.

Take a connected complex reductive affine algebraic group $G$, and also take $(X,\,S)$ as above.
Fix an integer $N_{i}\,\geq\, 2$ for every $1\,\leq\, i\,\leq\, n$. Assume that if $X\,\cong\, \mathbb{P}^{1}$,
then $n\,\geq\, 2$ and if $X\,\cong\, \mathbb{P}^{1}$ and $n\,=\,2$, then $N_{1}\,=\,N_{2}$. 

A {\it parabolic principal $G$--bundle} on $(X,\,S)$ is a complex manifold $\mathcal{E}_{G}$ together with a
holomorphic surjection $p\,:\,\mathcal{E}_{G}\,\longrightarrow\, X$ and a right holomorphic
action $\psi\,:\,\mathcal{E}_{G}\times G\,\longrightarrow\,
 \mathcal{E}_{G}$ of $G$ on $\mathcal{E}_{G}$ satisfying the following conditions:
\begin{itemize}
\item $p(\psi(x,g))\,=\,p(x)$ for all $x\,\in\, \mathcal{E}_{G}$ and $g\,\in\, G$,

\item for every $x\,\in\, X$ the action of $G$ on $p^{-1}(x)$ is transitive,

\item the restriction map $p\,:\,{\mathcal{E}_{G}}\big\vert_{p^{-1}(X\backslash S)}\, \longrightarrow\,
 X\backslash S$ is a holomorphic principal $G$--bundle and
 
\item for any $1\,\leq\, i\,\leq\, n$ and $x\,\in\, p^{-1}(x_{i})$, the stabilizer $G_x$ at $x$ for the
action $\psi$ of $G$ on $\mathcal{E}_{G}$ is a finite cyclic subgroup of $G$ whose order divides $N_{i}$.
\end{itemize}

Take a parabolic principal $G$--bundle $\mathcal{E}_{G}$ on $(X,\,S)$. Let
$\At_{*}(\mathcal{E}_{G})$ be the holomorphic vector bundle on $X$ defined by
\[\At_{*}(\mathcal{E}_{G})\ :=\ (p_{*}T\mathcal{E}_{G})^{G}.\]
It is called the {\it parabolic Atiyah bundle} of $\mathcal{E}_{G}$, and comes 
naturally with a parabolic structure over $S$; the parabolic weights at any $x_i$, $1\, \leq\, i\,
\leq\, n$, are integral multiples of $\frac{1}{N_{i}}$.

Let $T_{p}\,\subseteq\, TE_{G}$ be the relative tangent bundle of the projection $p$. Then the holomorphic 
vector bundle $\ad_{*}(\mathcal{E}_{G})$ on $X$ defined by 
$$
\ad_{*}(\mathcal{E}_{G})\ :=\ (p_{*}T_{p})^{G}
$$
is also naturally a parabolic vector bundle on $X$ with parabolic structure over $S$;
the parabolic weights over any $x_i$, $1\, \leq\, i\,
\leq\, n$, are integral multiples of $\frac{1}{N_{i}}$. The parabolic vector bundle
$\ad_{*}(\mathcal{E}_{G})$ is called the {\it parabolic adjoint bundle} of $\mathcal{E}_{G}$.

Let $\mathbb S$ be the effective divisor on $X$ defined by ${\mathbb S}\,=\,\Sigma_{1\leq i\leq n}x_{i}$. Let
$TX_*$ be the parabolic line bundle on $(X,\,S)$ whose underlying line bundle is
$TX(-{\mathbb S})\,=\, TX\otimes {\mathcal O}_X(-{\mathbb S})$ and the parabolic weight at $x_i$ is
$\frac{1}{N_{i}}$ for every $1\,\leq\, i\,\leq\, n$. 

There is the following natural short exact sequence of parabolic vector bundles:
\begin{equation}\label{Atpar}
0\ \longrightarrow\ \ad_{*}(\mathcal{E}_{G})\ \xrightarrow{\,\,\,{\iota_{G}}_{*}\,\,\,}\ 
\At_{*}(\mathcal{E}_{G})\ \xrightarrow{\,\,\,{\omega_{G}}_{*}\,\,\,}\ TX_{*}\ \longrightarrow\ 0,
\end{equation}
which is known as the {\it parabolic Atiyah exact sequence}.

A {\it parabolic connection} on $E_{G}$ is a splitting of the short exact sequence of parabolic vector bundles 
in \eqref{Atpar} via a holomorphic morphism of parabolic vector bundles. In other words, a parabolic connection
on $E_{G}$ is a holomorphic homomorphism of parabolic vector bundles
$$\delta\ :\ \At_{*}(\mathcal{E}_{G})\ \longrightarrow\ \ad_{*}(\mathcal{E}_{G})$$
such that $\delta\circ {\iota_{G}}_{*}\,=\, {\rm Id}_{\ad_{*}(\mathcal{E}_{G})}$; see Remark
\ref{rem0}.

A {\it parabolic Lie algebroid} is a pair $\mathcal{L}_{*}\,=\, (V_{*},\,\phi_{*})$
where $V_{*}$ is a parabolic vector bundle on $X$ with parabolic structure over $S$
such that the parabolic weights at each $x_i$, $1\, \leq\, i\, \leq\, n$, are integral
multiples of $\frac{1}{N_{i}}$, and $\phi_{*}\,:\,V_{*}\,\longrightarrow\, TX_{*}$ is a parabolic
homomorphism, such that the following two conditions hold: 
\begin{itemize}
\item $V_{*}$ is equipped with a $\mathbb{C}$--Lie algebra structure \[[-,\,-]\,:\,V_{*}\otimes_{\mathbb{C}}V_{*}
\,\longrightarrow\, V_{*}\] which is compatible with the parabolic structure, and

\item $[s,\, ft]\,=\,f[s,\,t]+\phi_{*}(s)(f)t$ for all locally defined holomorphic sections $s,\,t$ of $V$
and locally defined holomorphic functions $f$ on $X$.
\end{itemize}

Take a parabolic Lie algebroid $\mathcal{L}_{*}\,=\,(V_{*},\,\phi_{*})$ and a parabolic principal $G$--bundle 
$E_{G}$ as above. Consider the following surjective homomorphism of parabolic vector bundles
$$
{\psi_{G}}_{*}\ :\ V_{*}\oplus \At_{*}(\mathcal{E}_{G})\ \longrightarrow\ TX_{*},\ \ \,
(v,\,w)\ \longmapsto\ \phi_{*}(v)-{\omega_{G}}_{*}(w),
$$
where ${\omega_{G}}_{*}$ is the projection in \eqref{Atpar}. Define
$$\mathcal{A}_{*}(\mathcal{E}_{G})\ :=\ {\rm kernel}({\psi_{G}}_{*}).$$
Just as in Section \ref{2.3}, we obtain the following short exact sequence of parabolic vector 
bundles on $(X,\,S)$
\begin{equation}\label{par1}
0\ \longrightarrow\ \ad_{*}(\mathcal{E}_{G})\ \xrightarrow{\,\,\,{j_{G}}_{*}\,\,\,}\ 
\mathcal{A}_{*}(\mathcal{E}_{G})\ \xrightarrow{\,\,\,{\rho_{G}}_{*}\,\,\,}\
V_{*}\ \longrightarrow\ 0;
\end{equation}
the above homomorphism ${\rho_{G}}_{*}$ is the composition of homomorphisms 
\[\mathcal{A}_{*}(\mathcal{E}_{G})\ \hookrightarrow\ V_{*}\oplus \At_{*}(\mathcal{E}_{G})
\ \overset{\pi_{1}}\longrightarrow\ V_{*},\]
where $\pi_1$ is the natural projection to the first factor, and
${j_{G}}_{*}\, :=\, (0,\, {\iota_{G}}_{*})$.

A {\it parabolic $\mathcal{L}_{*}$--connection} on $\mathcal{E}_{G}$ is defined to be a splitting of the short 
exact sequence in \eqref{par1} of parabolic vector bundles via a morphism of parabolic vector bundles. In 
other words, a parabolic $\mathcal{L}_{*}$--connection on $\mathcal{E}_{G}$ is a holomorphic homomorphism of 
parabolic vector bundles
$$
\delta\ :\ \mathcal{A}_{*}(\mathcal{E}_{G})\ \longrightarrow\ \ad_{*}(\mathcal{E}_{G})
$$
such that $\delta\circ {j_{G}}_{*}\,=\, {\rm Id}_{\ad_{*}(\mathcal{E}_{G})}$; see Remark $\ref{rem0}$. 

Given $\{N_i\}_{i=1}^n$, there is a compact connected Riemann surface
$Y$, and a ramified Galois covering map $\sigma\,:\, Y\,\longrightarrow\, X$ such that the branch locus of
$\sigma$ is $S$, and for every $1\,\leq\, i\,\leq \,n$, the multiplicity of $\sigma$ at any
$y\,\in\, \sigma^{-1}(x_{i})$ is exactly $N_{i}$
(see \cite[Proposition 1.2.12]{Na}).

Fix a ramified Galois covering map $\sigma\,:\, Y\,\longrightarrow\, X$ satisfying the above condition.
Let
$$
\Gamma\ :=\ \Gal(\sigma)\ =\ \Aut(Y/X)
$$
be the Galois group, so we have $X\,=\, Y/\Gamma$.

The following notation will be used.
\begin{itemize}
\item $\mathscr{C}_{\Gamma}$ denotes the category of $\Gamma$--equivariant principal $G$--bundles on $Y$.

\item $\mathscr{C}$ denotes the category of parabolic principal $G$--bundles on $X$ with parabolic
structure on $S$.

\item $\mathscr{V}_{\Gamma}$ denotes the category of $\Gamma$--equivariant vector bundles on $Y$.

\item $\mathscr{V}$ denotes the category of parabolic vector bundles on $X$, with parabolic
structure over $S$ such that their parabolic weights at any $x_{i}\, \in\, S$ are
integral multiples of $\frac{1}{N_{i}}$. 
\end{itemize}

Let
$$\mathbb{A}_{*},\ \ad_{*}\,\ :\, \ \mathscr{C}\ \longrightarrow\ \mathscr{V}$$
be the functors that send a parabolic principal $G$--bundle to the corresponding parabolic Atiyah bundle and
parabolic adjoint vector bundle respectively. Similarly, the two functors
$$\mathbb{A},\ \ad\ :\ \mathscr{C}_{\Gamma}\ \longrightarrow\ \mathscr{V}_{\Gamma}$$
send a $\Gamma$--equivariant principal $G$--bundle to the corresponding
$\Gamma$--equivariant Atiyah bundle and $\Gamma$--equivariant adjoint bundle respectively.

\begin{theorem}\label{char}\mbox{}
\begin{enumerate}
\item There are natural equivalences of categories $\Phi\,:\,\mathscr{C}_{\Gamma}\,\overset{\sim}\longrightarrow
\, \mathscr{C}$ and $\Psi\,:\,\mathscr{V}_{\Gamma}\,\overset{\sim}\longrightarrow \,\mathscr{V}$. 

\item The following equalities of functors hold:
$$\Psi\circ\mathbb{A}\ = \ \mathbb{A}_{*}\circ\Phi\ \ \text{ and }\ \ \Psi\circ\mathrm{ad}\ =\ \mathrm{ad}_{*}\circ\Phi .$$

\item There is a natural isomorphism $\zeta\,:\,TX_{*}\, \overset{\sim}\longrightarrow\, \Psi(TY)$
given by the differential of $\sigma$. Moreover, for any $E_{G}\,\in\, \mathscr{C}_{\Gamma}$,
$\Psi$ takes the equivariant Atiyah exact sequence for 
$E_{G}$ to the parabolic Atiyah exact sequence for $\Phi(E_{G})\,\in\,\mathscr{C}$.

\item There is an equivalence of categories between the parabolic Lie algebroids on $(X,\,S)$ and the
$\Gamma$--equivariant Lie algebroids on $Y$. 

\item Take $E_{G}\,\in\, \mathscr{C}_{\Gamma}$ and a $\Gamma$--equivariant Lie algebroid
$(V,\,\phi)$ on $Y$. Giving a $\Gamma$--equivariant holomorphic $(V,\,\phi)$--connection on $E_{G}$ is equivalent
to giving a parabolic $(V_{*},\, \phi_{*})$--connection on $\mathcal{E}_{G}\,:=\,\Phi(E_{G})\,\in
\, \mathscr{C}$, where $(V_{*},\, \phi_{*})$ is the parabolic Lie algebroid on $X$ corresponding to
$(V,\, \phi)$.
\end{enumerate}
\end{theorem}

\begin{proof}
For (1) we refer the reader to the papers \cite{Bi2}, \cite{Bi3}, \cite{Bo1}, \cite{Bo2}, \cite{BBN1}, \cite{BBN2}.
Other statements are straightforward consequences of (1).
\end{proof}

\section{Equivariant Harder--Narasimhan reduction}\label{fourth}

As before, $G$ is a connected complex reductive affine algebraic group with Lie algebra $\mathfrak{g}$. Let 
$E_{G}$ be a holomorphic principal $G$--bundle on a compact connected Riemann surface $X$. A holomorphic 
reduction of structure group $E_{P}$ of $E_{G}$ to a parabolic subgroup $P\, \subset\, G$ is called {\it 
canonical} if the following two conditions hold:
\begin{enumerate}
\item If $L$ is the Levi quotient of $P$, then the corresponding holomorphic principal $L\,:=\,
P/R_u(P)$ bundle $E_{L}\,:=\,
E_{P}/R_u(P)$ is semistable, where $R_u(P)\, \subset\, P$ is the unipotent radical of $P$.

\item Fix a maximal torus $T$ and a Borel subgroup $B$ of $G$ such that $T\, \subset\, B\,\subseteq\, P$.
For any nontrivial character $\mathcal{X}$ of $P$, which is trivial on the center of $G$,
and which is a nonnegative linear combination of simple roots, the associated line bundle
$E_{P}\times^{\mathcal{X}}\mathbb{C}$ is of positive degree.
\end{enumerate}
(See \cite{Ra2}). See \cite[Ch. 13, \S~4]{Sp}, \cite[Ch. IV, \S~11]{Bo} for parabolic subgroups and
Levi quotients.

As in Section \ref{s2.4}, denote by $\mathfrak p$ (respectively, $\mathfrak{n}$)
the Lie algebra of $P$ (respectively, $R_{u}(P)$). The Lie algebra of the Levi quotient
$L\,:=\, P/R_u(P)$ will be denoted by $\mathfrak l$; so
\begin{equation}\label{levi}
\mathfrak{l}\ =\ \mathfrak{p}/\mathfrak{n}.
\end{equation}
Let $\mathcal{S}(P)$ denote the isomorphism classes of irreducible $P$--modules that occur as submodules
of the $P$--module $\mathfrak{n}/[\mathfrak{n},\,\mathfrak{n}]$.

\begin{theorem}[{\cite[Theorem 1]{AAB}}]\label{1.2}
Let $E_{G}$ be a holomorphic principal $G$--bundle over a compact connected Riemann surface $X$.
Then there is a parabolic subgroup $P\, \subset\, G$ unique up to conjugation and
a unique holomorphic reduction of structure group $E_P\, \subset\, E_{G}$ to $P$, such that the following
two conditions hold:
\begin{enumerate}
\item The holomorphic principal $L$--bundle $E_{L}$ obtained by extending the structure group from $E_{P}$
to $L$ via the natural projection $P\,\twoheadrightarrow\, P/R_{u}(P)\,=\,L$ is semistable.

\item For any $V\,\in\, \mathcal{S}(P)$ the associated holomorphic vector bundle $E_{P}\times^{P}V$ on $X$
has positive degree.
\end{enumerate}
\end{theorem}

We refer the reader to \cite{Ra1} for the notion of semistability of principal bundles where it was first 
introduced. The holomorphic reduction in Theorem \ref{1.2} is called the {\it Harder--Narasimhan reduction}, 
which evidently is canonical (an explicit construction of it appears in \cite[Lemma 2.11]{AB2}).

Let $X$ be a compact connected Riemann surface and $\Gamma\,\subseteq\, \Aut(X)$ be a finite subgroup. Let 
$E_{G}$ be a $\Gamma$--equivariant holomorphic principal $G$--bundle on $X$. It will be shown that the 
Harder--Narasimhan reduction of $E_{G}$ is $\Gamma$--equivariant. The action of $\Gamma$ on $E_G$ produces an 
action of $\Gamma$ on the adjoint vector bundle ${\rm ad}(E_G)$. Let
\begin{equation}\label{hn1}
0\,=\,F_{0}\,\subsetneq\, F_{1}\,\subsetneq\, \cdots \,\subsetneq\, F_{\ell-1}\,\subsetneq\, F_{\ell}
\,=\, {\rm ad}(E_G)
\end{equation}
be the Harder--Narasimhan filtration of ${\rm ad}(E_G)$. From the uniqueness of the Harder--Narasimhan
filtration it follows immediately that the action of $\Gamma$ on ${\rm ad}(E_G)$ preserves the
filtration in \eqref{hn1}.

As noted before, the Lie algebra $\mathfrak g$ admits a nondegenerate symmetric $G$--invariant
bilinear form. Fix a nondegenerate symmetric $G$--invariant
bilinear form $\mathbb{B}$ on $\mathfrak{g}$. $\mathbb{B}$ produces a nondegenerate symmetric bilinear form on
${\rm ad}(E_G)$; this nondegenerate symmetric bilinear form on ${\rm ad}(E_G)$ is
denoted by $\widetilde{\mathbb{B}}$. Note that for a subbundle $F\, \subset\, {\rm ad}(E_G)$, we have
\begin{equation}\label{perp}
F^{\perp}\ =\ ({\rm ad}(E_G)/F)^{*}.
\end{equation}

Define $W_{j}\,:=\,F_{\ell-j}^{\perp}$ for all $0\,\leq\, j\,\leq\, \ell$. We have
the following filtration of $\Gamma$--equivariant holomorphic subbundles of ${\rm ad}(E_G)$:
\begin{equation}\label{hn2}
0\,=\,W_{0}\,\subsetneq\, W_{1}\,\subsetneq\,\cdots\, \subsetneq\, W_{\ell-1}\,\subsetneq\, W_{\ell}
\,=\, {\rm ad}(E_G).
\end{equation}
Note that using \eqref{perp} and \eqref{hn1} it follows that \eqref{hn2} is the Harder--Narasimhan filtration of
${\rm ad}(E_G)$, and hence from the uniqueness of the Harder--Narasimhan filtration it
follows that the filtrations in \eqref{hn1} and \eqref{hn2} actually coincide. In other words, we have
$F_{j}\,=\,F_{\ell-j}^{\perp}$ for all $0\,\leq\, j\,\leq\, \ell$. Thus the Harder--Narasimhan filtration of
${\rm ad}(E_G)$ is of the form 
\begin{equation}\label{hn3}
0\,=\,E_{-l-1}\,\subsetneq\, E_{-l}\,\subsetneq\,\cdots,\, \subsetneq\, E_{-1}\,\subsetneq\, E_{0}
\,\subsetneq\,\cdots,\, \subsetneq\, E_{l-1}\,\subsetneq\, E_{l}\,=\, {\rm ad}(E_G)
\end{equation}
for some $l\,\geq\, 0$.

Assume that ${\rm ad}(E_G)$ is not semistable, or equivalently, the principal $G$--bundle $E_{G}$ is not
semistable (see \cite[Proposition 2.10]{AB2} for the equivalence of these statements).
Thus, we have $l\,\geq\, 1$ in \eqref{hn3}. Moreover, $E_{-j}^{\perp}\,=\,E_{j-1}$ for all $-l\,\leq\, j\,
\leq \, l+1$.

Note that since $E_{-1}^{\perp}\,=\,E_{0}$, the bilinear form $\widetilde{\mathbb{B}}$ on ${\rm ad}(E_G)$
induces a nondegenerate symmetric bilinear form on the quotient bundle $E_{0}/E_{-1}$, which, in turn,
produces an isomorphism
$$E_{0}/E_{-1}\ \cong\ (E_{0}/E_{-1})^{*},$$ in particular, we have
${\rm degree}(E_{0}/E_{-1})\,=\,0$.

Consider the homomorphism
\begin{equation}\label{ephi}
\phi\ :\ E_{0}\otimes E_{0}\ \longrightarrow\ {\rm ad}(E_G)/E_{0}
\end{equation}
defined by the composition of maps
\[E_{0}\otimes E_{0}\ \hookrightarrow\ {\rm ad}(E_G)\otimes {\rm ad}(E_G)\ \xrightarrow{\,\,\,[-,\,-]\,\,\,}\
{\rm ad}(E_G)\ \twoheadrightarrow\ {\rm ad}(E_G)/E_{0}.\]
Note that $$\mu_{min}(E_{0}\otimes E_{0})\ =\ 2\mu_{min}(E_{0})\ =\ \mu(E_{0}/E_{-1})\ =\ 0
$$
$$
>\ \mu(E_{1}/E_{0})\ =\ \mu_{max}({\rm ad}(E_G)/E_{0}).$$
Since $\mu_{min}(E_{0}\otimes E_{0})\, >\, \mu_{max}({\rm ad}(E_G)/E_{0})$, it follows immediately
that $\phi\,=\, 0$ (see \eqref{ephi}).
Thus, we have $[E_{0},\, E_{0}]\,\subseteq\, E_{0}$. In fact $E_{0}$ is a parabolic subalgebra bundle of
${\rm ad}(E_G)$ \cite[\S~10]{AB3}. 

Similarly, one argues that for all $-l\, \leq\, j\,\leq\, l$,
\begin{equation}\label{c}
[E_{-j},\ E_{-1}]\,\ \subseteq\,\ E_{-j-1}.
\end{equation}

Recall that ${\rm ad}(E_G)$ is the vector bundle associated to the principal $G$--bundle $E_G$ for the
adjoint action of $G$ on its Lie algebra $\mathfrak g$. This means that ${\rm ad}(E_G)$ is the quotient 
of $E_G\times {\mathfrak g}$, where two elements $(z_1,\, v_1),\, (z_2,\, v_2)\, \in\, E_G\times {\mathfrak g}$
are identified if there is an element $\gamma\, \in\, G$ such that $z_2\,=\, z_1\gamma$ and $v_2\,=\,
\text{Ad}(\gamma^{-1})(v_1)$. So fixing $x\, \in\, X$ and $z\, \in\, (E_{G})_{x}$, we get an isomorphism
$$\phi^{z}_{x}\ :\ \mathfrak{g}\ \overset{\sim}\longrightarrow\ {\rm ad}(E_G)_x$$
that sends any $v\, \in\, \mathfrak g$ to the equivalence class of $(z,\, v)$.

Let $\mathfrak{p}^{z}_{x}$ be the parabolic subalgebra $(\phi^{z}_{x})^{-1}((E_{0})_{x})
\, \subset\, \mathfrak{g}$. As $E_G$ is connected, the conjugacy class of the parabolic
subalgebra $(\phi^{z}_{x})^{-1}((E_{0})_{x})\, \subset\, \mathfrak{g}$ is independent of $z$ and $x$.

Fix a parabolic subalgebra
$$
\mathfrak{p}\ \subset\ \mathfrak{g}
$$
in the conjugacy class of $(\phi^{z}_{x})^{-1}((E_{0})_{x})$.

Let $x\,\in\, X$ and $z\,\in\, (E_{G})_{x}$. Choose $g^{z}_{x}\,\in\, G$ such that
$\Ad(g^{z}_{x})(\mathfrak{p})\,=\, \mathfrak{p}^{z}_{x}$ (which is defined above). Let
$P\, \subset\, G$ (respectively, $P^{z}_{x}\, \subset\, G$) be the parabolic subgroup whose
Lie algebra is $\mathfrak p$ (respectively, $\mathfrak{p}^{z}_{x}$).
Then clearly we have $g^{z}_{x}P{g^{z}_{x}}^{-1}\,=\,P^{z}_{x}$.

The Harder--Narasimhan reduction is constructed in \cite[Lemma 2.11]{AB2} using the section
\begin{equation}\label{es}
\sigma\ :\ X\ \longrightarrow\ E_{G}/P
\end{equation}
defined by $\sigma(x)\,=\, zg^{z}_{x}P$.
 
\begin{lemma}\label{leme}
The section $\sigma$ in \eqref{es} is $\Gamma$--equivariant.
\end{lemma}

\begin{proof}
Take any $\gamma \,\in\, \Gamma$. Let $T^{x}_{\gamma}\,:\,\text{ad}(E_G)_{x}\,\overset{\sim}\longrightarrow\,
\text{ad}(E_G)_{\gamma. x}$ be the isomorphism of Lie algebras given by the action of $\gamma$
on $\text{ad}(E_G)$ (which is induced by the action of $\gamma$ on $E_G$).
We have the identity \[\phi^{\gamma.z}_{\gamma.x}\ =\ T^{x}_{\gamma}\circ\phi^{z}_{x}.\] This, along with the
fact that $E_{0}$ is a $\Gamma$--equivariant subbundle of $\text{ad}(E_G)$, together imply
that \[\mathfrak{p}^{\gamma.z}_{\gamma.x}\ =\ {\phi^{\gamma.z}_{\gamma.x}}^{-1}((E_{0})_{\gamma.x})
\ =\ {\phi^{z}_{x}}^{-1}((E_{0})_{x})\ =\ \mathfrak{p}^{z}_{x}.\] Thus, we have
$$\Ad(g^{z}_{x})(\mathfrak{p})\,=\,\mathfrak{p}^{z}_{x}\,=\,\mathfrak{p}^{\gamma.z}_{\gamma.x}
=\Ad(g^{\gamma.z}_{\gamma.x})(\mathfrak{p}),$$ and hence $${g^{z}_{x}}^{-1}g^{\gamma.z}_{\gamma.x}\
\in\ P.$$ So, $\sigma(\gamma.x)\,=\,(\gamma.z).g^{\gamma.z}_{\gamma.x}P\,=\,(\gamma.z).g^{z}_{x}P
\,=\,\gamma.\sigma(x)$.
\end{proof}

{}From Lemma \ref{leme} it follows that the Harder--Narasimhan reduction $E_{P}\,\subseteq\, E_{G}$ is
in fact $\Gamma$--equivariant. From \eqref{c} it follows that $E_{-1}$ is a nilpotent ideal of $E_{0}$. Since
$E_{0}/E_{-1}$ has a nondegenerate symmetric bilinear form, we conclude that $E_{-1}$ is the nilpotent radical
of $E_{0}$. Therefore, since $E_{-1}^{\perp}\,=\,E_{0}$,
we see that $\mathfrak{p}^{\perp}$ is the nilpotent radical of $\mathfrak{p}$.

Note that there are natural isomorphisms of $\Gamma$--equivariant holomorphic
vector bundles $$E_{P}\times_{P}\mathfrak{p}^{\perp}\ \cong\ E_{-1},\ \ \ad(E_{P})
\ \cong\ E_{0},\ \ \ad(E_{G})/\ad(E_{P})\ \cong\ \text{ad}(E_G)/E_{0}.$$
Since $E_{G}$ is not semistable, we have $\mu({\rm ad}(E_G)/E_{0})\,<\,0$ and hence
$$H^{0}(X,\, \ad(E_{G})/\ad(E_{P}))\ =\ 0$$ (see \cite[Corollary 1]{AAB}).
Thus we have the following:

\begin{proposition}\label{3.3}
Let $X$ be a compact connected Riemann surface and $\Gamma\,\subseteq\, \emph{Aut}(X)$ a finite subgroup. Let
$G$ be a connected complex reductive affine algebraic group. Let $E_{G}$ be a non-semistable
$\Gamma$--equivariant holomorphic principal $G$--bundle on $X$. Then the Harder--Narasimhan reduction
$E_{P}\,\subset \,E_{G}$ of Theorem \ref{1.2} is $\Gamma$--equivariant and infinitesimally rigid, i.e.,
$H^{0}(X,\,\emph{ad}(E_{G})/\emph{ad}(E_{P}))\,=\,0$. 
\end{proposition}

\section{Criterion for equivariant holomorphic Lie algebroid connection}\label{final}

The notation of Section \ref{fourth} will be used. In the following theorem, ``equivariant" means 
$\Gamma$--equivariant.

\begin{theorem}\label{4.1}
Take a compact connected Riemann surface $X$ and also a connected complex reductive affine algebraic group
$G$. Let $E_{G}$ be an equivariant holomorphic principal $G$--bundle on $X$ equipped
with an equivariant holomorphic reduction of structure group of $E_{P}\,\subset\, E_{G}$, to some parabolic
subgroup $P\, \subset\, G$, that is equivariantly infinitesimally rigid. Take any equivariant holomorphic
Lie algebroid $(V,\,\phi)$ on $X$ that satisfies the condition that the anchor map $\phi$ is not
surjective. Then
the holomorphic principal $P$--bundle $E_{P}$ admits an equivariant holomorphic $(V,\,\phi)$--connection.
\end{theorem}

\begin{proof}
First note that if $\phi\,=\,0$, then the holomorphic vector bundle $\mathcal{A}(E_{P})$
for the equivariant principal $P$--bundle $E_P$ (see \eqref{d}) is ${\rm ad}(E_P)\oplus V$. In that case, the
natural inclusion map
$$
V\ \hookrightarrow\ {\rm ad}(E_P)\oplus V\ =\ \mathcal{A}(E_{P})
$$
define an equivariant holomorphic $(V,\,\phi)$--connection on the equivariant principal $P$--bundle $E_P$.

In view of the above observation it is assumed that $\phi\,\not=\, 0$.

Since $\phi\,\not=\, 0$, there is an effective divisor $D$ on $X$ such that $\phi(V)\,=\,TX(-D)\,:=\,
TX\otimes{\mathcal O}_X(-D)\, \subset\, TX$.
Note that $D$ is nonzero because $\phi$ is not surjective.
Since $\phi$ is equivariant, $\phi(V)\,=\,TX(-D)$ is an equivariant subbundle of $TX$. Consider the
natural inclusion map
\begin{equation}\label{eio}
\iota\ :\ TX(-D)\ \hookrightarrow\ TX.
\end{equation}
So
\begin{equation}\label{ei}
\mathcal{L}\ :=\ (TX(-D),\, \iota)
\end{equation}
is a holomorphic equivariant Lie algebroid.

Let
\begin{equation}\label{xi}
\xi\ :\ E_P\ \longrightarrow\ X
\end{equation}
be the natural projection.
Consider the Atiyah exact sequence for the principal $P$--bundle $E_{P}$: 
\begin{equation}\label{1}
0\ \longrightarrow\ \ad(E_{P})\ \xrightarrow{\,\,\,\iota_{P}\,\,\,}\ \At(E_{P})\ =\
(\xi_* TE_P)^P \ \xrightarrow{\,\,\, \omega_{P}\,\,\,}\ TX\ \longrightarrow\ 0.
\end{equation}
Note that \eqref{1} is equivariant because $E_{P}$ is so. 

Our aim is to show that there exists an equivariant holomorphic ${\mathcal O}_X$--linear map
\begin{equation}\label{de}
\delta^{'}\ :\ TX(-D)\ \longrightarrow\ \At(E_{P})
\end{equation}
such that $\omega_{P}\circ\delta^{'}\,=\,\iota$, where $\omega_{P}$ and $\iota$ are constructed
in \eqref{1} and \eqref{eio} respectively. 

Let $L\,=\, P/R_u(P)$ be the Levi quotient of $P$, where $R_u(P)\, \subset\, P$ is the
unipotent radical. To prove \eqref{de}, first consider the principal $L$--bundle on $X$
\begin{equation}\label{eu}
E_{L}\ :=\ E_P/R_u(P),
\end{equation}
where the action of $R_u(P)$ on $E_P$ is the restriction to $R_u(P)\, \subset\, P$ of the natural
action of $P$ on the principal $P$--bundle $E_P$. We note that $E_L$ coincides with the principal
$L$--bundle on $X$ given by the extension of the structure group of the principal $P$--bundle
$E_P$ for the quotient map
$P\, \longrightarrow\, P/R_u(P)\,=\, L$. The principal $L$--bundle $E_L$ has an equivariant structure
given by the equivariant structure of the principal $P$--bundle $E_P$. More precisely, the
actions of $P$ and $\Gamma$ on $E_P$ commute, and hence the quotient $E_P/R_u(P)$ in \eqref{eu} has
an induced action of $\Gamma$.

Let
\begin{equation}\label{dp}
\pi\ :\ E_{P}\ \longrightarrow\ E_P/R_u(P)\ =\ E_{L}
\end{equation}
be the quotient map. Let $\At(E_{L})$ be the Atiyah bundle for $E_{L}$;
this holomorphic vector bundle $\At(E_{L})$ is equivariant because $E_{L}$ is so. Let
\begin{equation}\label{ol}
\omega_L\ :\ \text{At}(E_L)\ \longrightarrow\ TX
\end{equation}
be the canonical map of equivariant holomorphic vector bundles just as $\omega_{P}$ in \eqref{1}.
As in Section \ref{2.3}, define
\begin{equation}\label{dp2a}
\mathcal{A}(E_P)\ :=\ (\omega_P)^{-1}(TX(-D))\subset\textrm{At}(E_P)
\end{equation}
where $\omega_P$ is the homomorphism in \eqref{1}, and
\begin{equation}\label{dp2b}
\mathcal{A}(E_L)\ :=\ (\omega_L)^{-1}(TX(-D))\subset\textrm{At}(E_L),
\end{equation}
where $\omega_L$ is the homomorphism in \eqref{ol}.

Consider the differential
\begin{equation}\label{ed}
d\pi\ :\ TE_{P}\ \longrightarrow\ \pi^*TE_{L}
\end{equation}
of the projection $\pi$ in \eqref{dp}. The natural action of $P$ on $E_P$ produces an action
of $P$ on $TE_P$. The natural action of $L$ on $E_L$ and the quotient map $P\, \longrightarrow\, P/R_u(P)\,=\, L$
together produce an action of $P$ on $E_L$; in other words, $P$ acts on $E_P$ through the quotient $L$
of $P$. The map $\pi$ in \eqref{dp} is evidently $P$--equivariant. Therefore, $\pi^*TE_{L}$ in \eqref{ed}
is $P$--equivariant. The differential $d\pi$ in \eqref{ed} is evidently $P$--equivariant.
The actions of $P$ on $TE_{P}$ (respectively, $\pi^*TE_{L}$) produces an action of $P$ on the direct image
$\xi_*TE_{P}$ (respectively, $\xi_*\pi^*TE_{L}$) by the projection $\xi$ in \eqref{xi}. The direct image by $\xi$
\begin{equation}\label{ed2}
\xi_* d\pi\ :\ \xi_* TE_{P}\ \longrightarrow\ \xi_*\pi^*TE_{L}
\end{equation}
is evidently $P$--equivariant. Taking invariants, for the actions of $P$, of \eqref{ed2} we have a homomorphism
\begin{equation}\label{ea0}
d'\pi\ :\ \text{At}(E_P) \ :=\ (\xi_* TE_{P})^P \ \longrightarrow\ (\xi_*\pi^*TE_{L})^P.
\end{equation}

Let
$$
\xi'\ :\ E_L\ \longrightarrow\ X
$$
be the natural projection. It is straightforward to check that there is natural homomorphism
\begin{equation}\label{en}
(\xi'_* TE_L)^L \ \longrightarrow\ (\xi_*\pi^*TE_{L})^P
\end{equation}
because the action of $P$ on $E_L$ factors through the quotient map $P\, \longrightarrow\, L$. From the
fact that the action of $P$ on $E_L$ factors through the quotient map $P\, \longrightarrow\, L$ it also
follows that the homomorphism in \eqref{en} is an isomorphism. Therefore, we have
$$
(\xi_*\pi^*TE_{L})^P \ = \ (\xi'_* TE_L)^L \ = \ \text{At}(E_L).
$$
Consequently, from \eqref{ea0} we have a homomorphism
\begin{equation}\label{ea}
d'\pi\ :\ \text{At}(E_P) \ :=\ (\xi_* TE_{P})^P \ \longrightarrow\ \text{At}(E_L)
\end{equation}
between the Atiyah bundles.

The homomorphism $d'\pi$ in \eqref{ea} restricts to a homomorphism
$$
\gamma\ :\ \mathcal{A}(E_P)\ \longrightarrow\ \mathcal{A}(E_L)
$$
(see \eqref{dp2a} and \eqref{dp2b}). The homomorphism $d'\pi$ in \eqref{ea} restricts to a homomorphism
$$
\beta\ :\ \text{ad}(E_P)\ \longrightarrow\ \text{ad}(E_L).
$$
We have the following commutative diagram of equivariant holomorphic vector bundles:
\begin{center}
\begin{equation}\label{s}
\begin{tikzpicture}[baseline={([yshift=-0.5ex]current bounding box.center)}]
\matrix(m)[matrix of math nodes,
row sep=2em, column sep=2.5em,
text height=1.5ex, text depth=0.25ex]
{
& 0 & 0 & & \\
& {\rm kernel}(\beta) & {\rm kernel}(\gamma)& & \\
0 & \ad(E_{P}) & \mathcal{A}(E_{P}) & TX(-D) & 0 \\
0 & \ad(E_{L}) & \mathcal{A}(E_{L}) & TX(-D) & 0\\
};
\path[->]
(m-3-2) edge node[above] {$j_{P}$} (m-3-3);
\path[->]
(m-3-3) edge node[above] {$\rho_{P}$} (m-3-4);
\path[->]
(m-4-2) edge node[below] {$j_{L}$} (m-4-3);
\path[->]
(m-4-3) edge node[below] {$\rho_{L}$} (m-4-4);
\path[->]
(m-3-4) edge node[right] {$id$} (m-4-4);
\path[->]
(m-1-3) edge node[right] {} (m-2-3);
\path[->]
(m-1-2) edge node[right] {} (m-2-2);
\path[->]
(m-3-3) edge node[right] {$\gamma$} (m-4-3);
\path[->]
(m-3-2) edge node[right] {$\beta$} (m-4-2);
\path[->]
(m-3-1) edge node[right] {} (m-3-2);
\path[->]
(m-4-1) edge node[right] {} (m-4-2);
\path[->]
(m-3-4) edge node[right] {} (m-3-5);
\path[->]
(m-4-4) edge node[right] {} (m-4-5);
\path[->]
(m-2-2) edge node[right] {$\alpha^{'}$} (m-3-2);
\path[->]
(m-2-3) edge node[right] {$\alpha$} (m-3-3);
\path[->]
(m-2-2) edge node[above] {$\sigma$} (m-2-3);
\path[->]
(m-2-2) edge node[below] {$\cong$} (m-2-3);
\end{tikzpicture}
\end{equation}
\end{center}
Here $\rho_{P}$ (respectively, $\rho_{L}$) is given by $\omega_P$ (respectively, $\omega_L$) in
\eqref{dp2a} (respectively, \eqref{dp2b}) (see \eqref{p}), while $j_{P}$ (respectively, $j_{L}$) is
the natural inclusion map (see \eqref{r}).

Since $D\,\neq\, 0$, the equivariant Lie algebroid $\mathcal{L}$ in \eqref{ei} is non-split (see
\cite{Bi1} for definition of non-split Lie algebroids).
Consequently, the equivariant $L$--bundle $E_{L}$ admits an equivariant 
holomorphic connection for the Lie algebroid $\mathcal{L}$ \cite[Theorem 1.1]{Bi1}.

Let $$\delta\ :\ TX(-D)\ \longrightarrow\ \mathcal{A}(E_{L})$$
be an equivariant ${\mathcal O}_X$--linear homomorphism such that
\begin{equation}\label{l}
\rho_{L}\circ\delta\ =\ {\rm Id}_{TX(-D)},
\end{equation}
where $\rho_L$ is the homomorphism in \eqref{s}. Define
$$
B(\delta)\ :=\ \gamma^{-1}(\delta(TX(-D)))\ \subset\ \mathcal{A}(E_{P}),
$$
where $\gamma$ is the surjective map in \eqref{s}. So $B(\delta)$ fits in the following
commutative diagram of equivariant holomorphic vector bundles:
\begin{center}
\begin{equation}\label{a}
\begin{tikzpicture}[baseline={([yshift=-0.5ex]current bounding box.center)}]
\matrix(m)[matrix of math nodes,
row sep=2em, column sep=2.5em,
text height=1.5ex, text depth=0.25ex]
{
B(\delta) & \mathcal{A}(E_{P}) \\
TX(-D) & \mathcal{A}(E_{L}) \\
};
\path[->]
(m-1-1) edge node[above] {$\widetilde{\delta}$} (m-1-2);
\path[->]
(m-1-1) edge node[left] {$\widetilde{\gamma}$} (m-2-1);
\path[->]
(m-2-1) edge node[below] {$\delta$} (m-2-2);
\path[->]
(m-1-2) edge node[right] {$\gamma$} (m-2-2);
\end{tikzpicture}
\end{equation}
\end{center}
As $\gamma$ is surjective, so is $\widetilde{\gamma}$; as $\delta$ is injective, so is $\widetilde{\delta}$.
Furthermore, there is a natural isomorphism of equivariant holomorphic vector bundles 
${\rm kernel}(\widetilde{\gamma})\, \stackrel{\cong}{\longrightarrow}\, {\rm kernel}(\gamma)$. On the other hand,
from \eqref{s} it follows that ${\rm kernel}(\gamma)\, \stackrel{\cong}{\longrightarrow}\, {\rm kernel}(\beta)$.
Combining these two isomorphisms we have an isomorphism
$$
\eta\ :\ {\rm kernel}(\widetilde{\gamma})\ \stackrel{\cong}{\longrightarrow}\ {\rm kernel}(\beta),
$$ which fits in the following commutative diagram of equivariant holomorphic vector bundles:
\begin{center}
\begin{tikzpicture}
\matrix(m)[matrix of math nodes,
row sep=2em, column sep=2.5em,
text height=1.5ex, text depth=0.25ex]
{
{\rm kernel}(\widetilde{\gamma}) & {\rm kernel}(\beta) \\
B(\delta)& \mathcal{A}(E_{P}) \\
};
\path[->]
(m-1-1) edge node[above] {$\eta$} (m-1-2);
\path[->]
(m-1-1) edge node[below] {$\cong$} (m-1-2);
\path[->]
(m-1-1) edge node[left] {$u$} (m-2-1);
\path[->]
(m-2-1) edge node[below] {$\widetilde{\delta}$} (m-2-2);
\path[->]
(m-1-2) edge node[right] {$\alpha\circ\sigma=j_{P}\circ\alpha^{'}$} (m-2-2);
\end{tikzpicture}
\end{center}

\textbf{Claim 1:}\, The short exact sequence of equivariant holomorphic vector bundles
\begin{equation}\label{2}
0\ \longrightarrow\ {\rm kernel}(\beta)\ \xrightarrow{\,\,\,u\circ\eta^{-1}\,\,\,}\ B(\delta)
\ \overset{\widetilde{\gamma}}\longrightarrow\ TX(-D)\ \longrightarrow\ 0
\end{equation}
splits holomorphically which is also an equivariant splitting.

{\it Proof of Claim 1:}\, As in \eqref{levi}, $\mathfrak{n}$ denotes the Lie algebra of the
unipotent radical $R_u(P)$ in \eqref{eu}. Let
$$
\mathcal{R}(E_{P})\ :=\ E_P\times^P \mathfrak{n}\ \longrightarrow\ X
$$
be the vector bundle associated to the principal $P$--bundle $E_P$ for the adjoint action
of $P$ on $\mathfrak{n}$. Using \eqref{levi} it follows that there is a natural isomorphism of equivariant
holomorphic vector bundles 
\begin{equation}\label{3}
\mathcal{R}(E_{P})\,\ \cong\,\ {\rm kernel}(\beta).
\end{equation}

Let $\varphi_1\, \in\, H^1(X,\, \text{Hom}(TX(-D),\, {\rm kernel}(\beta)))^\Gamma\,=\,
H^1(X,\, {\rm kernel}(\beta)\otimes K_X\otimes{\mathcal O}_X(D))^\Gamma$ be the extension class for \eqref{2}
(recall that \eqref{2} is $\Gamma$--equivariant). The line bundle $K_X\otimes{\mathcal O}_X(D)$ will be
denoted by $K_X(D)$. Using \eqref{3}, the cohomology class $\varphi_1$ corresponds to a cohomology class
\begin{equation}\label{vp}
\varphi\,\ \in\, \ H^1(X,\ \mathcal{R}(E_{P})\otimes K_X(D))^\Gamma.
\end{equation}
By Serre duality,
\begin{equation}\label{5}
H^{1}(X,\, \mathcal{R}(E_{P})\otimes K_{X}(D))^{\Gamma}\ \cong\ (H^{0}(X,\,
\mathcal{R}(E_{P})^{*}\otimes \mathcal{O}_{X}(-D))^\Gamma)^{*}.
\end{equation}
As $D$ is an effective divisor, there is a natural inclusion map 
\begin{equation}\label{6} 
H^{0}(X,\, \mathcal{R}(E_{P})^{*}\otimes\mathcal{O}_{X}(-D))^{\Gamma}\ \hookrightarrow\
H^{0}(X,\ \mathcal{R}(E_{P})^{*})^{\Gamma}.
\end{equation}

Since the reduction $E_P\, \subset\, E_G$ is equivariantly infinitesimally rigid, from \eqref{6} and
\eqref{er2} we conclude that
$$
H^{0}(X,\, \mathcal{R}(E_{P})^{*}\otimes\mathcal{O}_{X}(-D))^{\Gamma}\ \,=\ \, 0.
$$
So from \eqref{5},
$$
H^{1}(X,\ \mathcal{R}(E_{P})\otimes K_{X}(D))^{\Gamma}\,\ =\, \ 0.
$$
In particular, we have $\varphi\,=\, 0$ (see \eqref{vp}). Thus the short exact sequence
in \eqref{2} splits holomorphically which is also an equivariant splitting.
This proves \textbf{Claim 1}.

Consequently, there is a map of equivariant holomorphic 
vector bundles $$\widehat{\theta}\,:\, TX(-D)\, \longrightarrow\, B(\delta)$$ such that 
\begin{equation}\label{o}
\widetilde{\gamma}\circ\widehat{\theta}\,\ =\,\ {\rm Id}_{TX(-D)},
\end{equation}
where $\widetilde{\gamma}$ is the homomorphism in \eqref{2}.

Let $\phi_{P}\,:\,\mathcal{A}(E_{P})\,\longrightarrow\, \At(E_{P})$ be the map of equivariant holomorphic
vector bundles as in \eqref{q} (for the equivariant holomorphic $P$--bundle $E_{P}$ and the
equivariant holomorphic Lie algebroid $\mathcal{L}$ in \eqref{ei}). Define a map of equivariant holomorphic
vector bundles $\delta^{'}\,:\, TX(-D)\, \longrightarrow\, \At(E_{P})$ by 
\[\delta^{'}\ :=\ \phi_{P}\circ\widetilde{\delta}\circ\widehat{\theta}.\] Thus:
\[\omega_{P}\circ\delta^{'}
=\omega_{P}\circ\phi_{P}\circ\widetilde{\delta}\circ\widehat{\theta}
=\iota\circ\rho_{P}\circ\widetilde{\delta}\circ\widehat{\theta}\]
\[=\iota\circ\rho_{L}\circ\gamma\circ\widetilde{\delta}\circ\widehat{\theta}~\,\,\,\,[\rho_{L}\circ\gamma=\rho_{P}~
\from~\eqref{s}]\]
\[=\iota\circ\rho_{L}\circ\delta\circ\widetilde{\gamma}\circ\widehat{\theta}~\,\,\,\,
[\delta\circ\widetilde{\gamma}=\gamma\circ\widetilde{\delta}~\from~\eqref{a}]\]
\[=\iota\circ\widetilde{\gamma}\circ\widehat{\theta}~\,\,\,\,[\rho_{L}\circ\delta=
{\rm Id}_{TX(-D)}~\,\, \from~\eqref{l}]\]
\[=\iota~\,\,\,\,[\widetilde{\gamma}\circ\widehat{\theta}={\rm Id}_{TX(-D)}~\from~\eqref{o}].\]
Note that the above equality $\omega_{P}\circ\phi_{P}\circ\widetilde{\delta}\circ\widehat{\theta}
=\iota\circ\rho_{P}\circ\widetilde{\delta}\circ\widehat{\theta}$ is implied by the equality
$\omega_{P}\circ\phi_{P}=\iota\circ\rho_{P}$ which follows from \eqref{d} applied to $(E_{P},\mathcal{L})$.
This proves \eqref{de}.

In view of Remark \ref{rem}, the composition of homomorphisms $\delta^{'}\circ\phi$, where $\delta'$ is the
homomorphism in \eqref{de} and $\phi$ is the anchor map (see the statement of Theorem \ref{4.1}), gives
an equivariant holomorphic $(V,\,\phi)$--connection on $E_P$.
\end{proof}

\begin{corollary}\label{imp}
Let $X$ be a compact connected Riemann surface and $G$ a connected complex reductive affine algebraic group.
Let $\Gamma$ be a finite group of biholomorphisms of $X$ and $E_{G}$ a $\Gamma$--equivariant holomorphic 
principal $G$--bundle on $X$ such that $E_G$ is not semistable. Take any $\Gamma$--equivariant holomorphic Lie 
algebroid $(V,\, \phi)$ on $X$ that satisfies the condition that the anchor map
$\phi$ not surjective. Then the Harder--Narasimhan reduction $E_{P}\,\subset\, 
E_{G}$ of $E_G$ admits a $\Gamma$--equivariant holomorphic $(V,\, \phi)$--connection.
\end{corollary} 

\begin{proof}
This follows directly from Theorem \ref{4.1} and Proposition \ref{3.3}.
\end{proof}

\section{Criterion for parabolic Lie algebroid connection}

As before, take a pair $(X,\,S)$, where $X$ is a compact connected Riemann surface and $S$ is a finite set of marked 
points in $X$. Fix a ramified Galois covering $\sigma\,:\,Y\,\longrightarrow\, X$ of $X$ with finite Galois group 
$\Gamma\,=\,\Gal(\sigma)$. Take a connected complex reductive affine algebraic group $G$. Recall
$\mathscr{C}$, $\mathscr{C}_{\Gamma}$, $\mathscr{V}$ and $\mathscr{V}_{\Gamma}$ in Section \ref{third}.

Let $P$ be any parabolic subgroup of $G$ and $\mathcal{E}_{G}\,\in\,\mathscr{C}$ a parabolic principal 
$G$--bundle. Denote
\begin{equation}\label{pg}
E_{G}\ :=\ \Phi^{-1}(\mathcal{E}_{G})\ \in\ \mathscr{C}_{\Gamma},
\end{equation}
where $\Phi^{-1}$ is the quasi-inverse of $\Phi$ in Theorem \ref{char}(i).

Let $E_{P}\subset E_{G}$ be any $\Gamma$--equivariant holomorphic reduction of structure group to $P$
of the $\Gamma$--equivariant principal $G$--bundle $E_{G}$ in \eqref{pg}. Consequently,
$\ad(E_{P})\,\subset\, \ad(E_{G})$ is a $\Gamma$--equivariant subbundle.

Define $\mathcal{E}_{P}\,:=\, \Phi(E_{P})$, where $\Phi$ is the same correspondence as in
Theorem \ref{char}(i) but for $P$ (note that Theorem \ref{char} is independent of the reductivity
assumption on the structure group of the principal bundle).

We call $\mathcal{E}_{P}$ to be a {\it parabolic reduction of structure group of $\mathcal{E}_{G}$ to the
parabolic subgroup $P$}. As $\Psi$ in Theorem $\ref{char}(i)$ is an equivalence of
categories, the second functorial isomorphism of Theorem $\ref{char}(ii)$ shows that there is an inclusion of parabolic vector bundles: 
 \[\ad_{*}(\mathcal{E}_{P})\,\ \hookrightarrow\,\ \ad_{*}(\mathcal{E}_{G}).\] 
It is known that $\Psi$ is precisely the functor $(\sigma_{*})^{\Gamma}$ (see for instance \cite{Bi3}) which
first pushes forward a $\Gamma$--equivariant holomorphic vector bundle on $Y$ to $X$ via $\sigma$ and then
takes the $\Gamma$--invariant subsheaf. As $\Gamma$ is a finite group, $(\sigma_{*})^{\Gamma}$ is an exact
functor from the category of coherent $\Gamma$--sheaves on $Y$ to the
coherent sheaves on $X$. Therefore, for any $k\,\geq\, 0$, we have
\begin{equation}\label{coh}
H^{k}(Y,\ \ad(E_{G})/\ad(E_{P}))^{\Gamma}\,\ \cong\,\ H^{k}(X,\
\ad_{*}(\mathcal{E}_{G})/\ad_{*}(\mathcal{E}_{P})).\end{equation}
In particular we have $E_{P}\,\subset\, E_{G}$ to be equivariantly infinitesimally rigid
(see \eqref{er}) if and only if $H^{0}(X,\,\ad_{*}(\mathcal{E}_{G})/\ad_{*}(\mathcal{E}_{P}))\,=\,0$.

If $H^{0}(X,\,\ad_{*}(\mathcal{E}_{G})/\ad_{*}(\mathcal{E}_{P}))\,=\,0$,
we call the parabolic reduction of structure group $\mathcal{E}_{P}$ of $\mathcal{E}_{G}$ to $P$
to be {\it parabolically infinitesimally rigid}. 

\begin{theorem}\label{MAIN}
Take a compact connected Riemann surface $X$ with a finite set of marked points $S\, \subset\, X$.
Take a connected complex reductive affine algebraic group $G$ and a parabolic subgroup $P\, \subset\, G$.
Let $\mathcal{E}_{G}$ be a parabolic principal $G$--bundle on $(X,\,S)$ equipped with a
parabolic reduction of structure group $\mathcal{E}_{P}\, \subset\, \mathcal{E}_{G}$ of
$\mathcal{E}_{G}$ to $P$ such that the reduction is parabolically infinitesimally rigid.
Then for any parabolic Lie algebroid $(V_{*},\, \phi_{*})$ on $(X,\,S)$ such that
$\phi_{*}$ is not surjective, there is a holomorphic parabolic $(V_{*},\, \phi_{*})$--connection
on $\mathcal{E}_{P}$.
\end{theorem}

\begin{proof}
Recall Theorem \ref{char}. There is a ramified Galois covering map $\sigma\,:\, Y\,\longrightarrow\, X$,
whose branch locus is $S$, such that $\mathcal{E}_{G}$ corresponds to a
$\Gamma$--equivariant holomorphic principal $G$--bundle $E_G$ on $Y$, where $\Gamma\,=\, \text{Gal}(\sigma)$
is the Galois group for $\sigma$. Let $E_{P}\,\subset\, E_{G}$ be the $\Gamma$--equivariant holomorphic
reduction of structure group of $E_{G}$ to $P$ given by reduction $\mathcal{E}_{P}\, \subset\, \mathcal{E}_{G}$ of
$\mathcal{E}_{G}$ to $P$. Since the reduction $\mathcal{E}_{P}\, \subset\, \mathcal{E}_{G}$
is parabolically infinitesimally rigid, from \eqref{coh} we conclude that the reduction
$E_{P}\,\subset\, E_{G}$ is equivariantly infinitesimally rigid.

Let $(V,\,\phi)$ be the $\Gamma$--equivariant holomorphic Lie algebroid on $Y$ that corresponds to
the parabolic Lie algebroid $(V_{*},\, \phi_{*})$. It should be clarified that we can choose the
ramified Galois covering map $\sigma$ in such a way that any given finite set of parabolic bundles on
$(X,\,S)$ correspond to equivariant bundles on $Y$. To see this, take ramified Galois coverings for
the individual parabolic  bundles in the given finite set of parabolic
bundles. Take the fiber product of
these finitely many ramified Galois coverings, and then take its normalization. The resulting
ramified Galois covering satisfies the condition that each of the parabolic  bundles in the
given finite set corresponds to an equivariant  bundle on the covering. This shows that $\mathcal{E}_{G}$ and $(V_{*},\phi_{*})$ can simultaneously be realized as an equivariant principal $G-$bundle and an equivariant holomorphic Lie algebroid respectively.

We note that $\phi$ is not surjective because $\phi_{*}$ is not surjective.

By Theorem \ref{4.1} we know that there is a $\Gamma$--equivariant holomorphic $(V,\,\phi)$--connection on 
$E_{P}$. Therefore, applying Theorem \ref{char}(v) for the complex linear algebraic group $P$, it follows 
that there is a parabolic $(V_{*},\, \phi_{*})$--connection on $\mathcal{E}_{P}$.
\end{proof}

Corollary \ref{imp} gives the following:

\begin{corollary}\label{main}
Take a compact connected Riemann surface $X$ with a finite set of marked points $S\, \subset\, X$.
Take a connected complex reductive affine algebraic group $G$ and a parabolic subgroup $P\, \subset\, G$.
Let $\mathcal{E}_{G}$ be a parabolic principal $G$--bundle on $(X,\,S)$ which is not semistable. Let
\begin{equation}\label{hnr}
\mathcal{E}_{P}\ \subset\ \mathcal{E}_{G}
\end{equation}
be the parabolic 
Harder--Narasimhan reduction of structure group of $\mathcal{E}_{G}$.
Then for any parabolic Lie algebroid $(V_{*},\, \phi_{*})$ on $(X,\,S)$ for which $\phi_{*}$ is not surjective, 
there is a holomorphic parabolic $(V_{*},\, \phi_{*})$--connection on $\mathcal{E}_{P}$.
\end{corollary}

\begin{proof}
{}From Theorem \ref{char} we know that there is a ramified Galois covering map
$\sigma\,:\, Y\,\longrightarrow\, X$, whose branch locus is $S$, such that $\mathcal{E}_{G}$ corresponds to a
$\Gamma$--equivariant holomorphic principal $G$--bundle $E_G$ on $Y$, where $\Gamma\,=\, \text{Gal}(\sigma)$
is the Galois group for $\sigma$. Let
$$
E_Q\ \subset\ E_G
$$
be the Harder--Narasimhan reduction of $E_G$. It is known that the equivariant principal $Q$--bundle $E_Q$
corresponds to $\mathcal{E}_P$ in \eqref{hnr} \cite[Proposition 4.1]{BBN1}, \cite[Theorem 3.14]{BBN2}. In
particular, $Q$ is conjugate to $P$.

In view of the above, as in the proof of Theorem \ref{MAIN}, using the ramified Galois covering map
$\sigma\,:\, Y\,\longrightarrow\, X$, the result is deduced from Corollary \ref{imp}. 
\end{proof}

\section*{Acknowledgements}
We thank the referee for helpful comments to improve the exposition.
The first-named author acknowledges the support of a J. C. Bose Fellowship (JBR/2023/000003).

\section*{Declaration}
No data were used or generated in this project. The authors do not have any conflict of interests.

\end{document}